\documentclass[11pt]{amsart}
\usepackage{amsfonts}
\usepackage{amssymb}
\usepackage{amsthm}
\usepackage{enumerate}
\usepackage{xcolor}
\usepackage{diagbox}

\newtheorem{theorem}{\textbf{Theorem}}

\newtheorem{lemma}{\textbf{Lemma}}

\theoremstyle{remark}

\numberwithin{equation}{section}

\begin{document}

	\title[$(q,k)$-Generalized Fibonacci Numbers]{On the Sequences of $(q,k)$-Generalized Fibonacci Numbers}
	
	\author{Gersica Freitas}
	\address{Campus de Engenharia e Ciências Agrárias, Universidade Federal de Alagoas, Rio Largo - AL, Brazil}
	\email{gersica.freitas@ceca.ufal.br}
	
	\author{Alessandra Kreutz}
	\address{Instituto Federal de Brasília, Campus Taguatinga, Brasília - DF, Brazil}
	\email{alessandra.kreutz@ifb.edu.br}
	
	\author{Jean Lelis}
	\address{Faculdade de Matemática/ICEN/UFPA, Belém - PA, Brazil.}
	\email{jeanlelis@ufpa.br}
	
	\author{Elaine Silva}
	\address{Instituto de Matemática, Universidade Federal de Alagoas, Maceió - AL, Brazil}
	\email{elaine.silva@im.ufal.br}


	
	
	\begin{abstract}
		In this paper, we consider the new family of recurrence sequences of $(q,k)$-generalized Fibonacci numbers. These sequences naturally extend the well-known sequences of $k$-generalized Fibonacci numbers and generalized $k$-order Pell numbers. We shall obtain a Binet-style formula and study the asymptotic behavior of dominant root of characteristic equation. Moreover, we shall prove some auxiliary results about these sequences. In particular, we characterize the first $(q,k)$-generalized Fibonacci numbers in terms of binary sequences. 
		
	\end{abstract}
	
	\maketitle
	
	\section{Introduction}
	
	Recurrence sequences are often the subject of study in the literature, for instance, the Fibonacci sequence and its generalizations have been widely studied due to 
	interesting results about their \textit{Binet-style} formula and their asymptotic behavior.
	The study of recurrence sequences has implications in many areas of research such as Diophantine equations, combinatorial problems and others. 
	
    The aim of this paper is to prove some properties associated with a new family of generalized recurrence sequences named \textit{sequences of $(q,k)$-generalized Fibonacci numbers} or simply \textit{$(q,k)$-bonacci numbers} which is given recurrently by 
		\begin{equation}\label{eq04}
			F_{q,n}^{(k)}=qF_{q,n-1}^{(k)}+F_{q,n-2}^{(k)}+\cdots+F_{q,n-k}^{(k)}\qquad\forall\quad n\geq 2
		\end{equation}
		with the $k$ initial conditions given by $F_{q,-(k-2)}^{(k)}=F_{q,-(k-3)}^{(k)}=\cdots =F_{q,0}^{(k)}=0$ and $F_{q,1}^{(k)}=1$.
	
	First of all, let us note that the family of the sequences of $(q,k)$-generalized Fibonacci numbers extends well-known recurrence sequences.
	
	For $q=1$ and $k=2$ we have the sequence of Fibonacci numbers $(F_n)_{n\geq 0}$ defined recurrencely by $F_{n+1}=F_n+F_{n-1}$ with initial conditions $F_0=0$ and $F_1=1$. Such sequence have been studied and generalized by many authors, see \cite{bravo,dresden,koshy,miller}. Some of these generalizations have received greater prominence over the years. Among these, it is worth mentioning here that the usually called sequences of $k$-generalized Fibonacci numbers or $k$-bonacci number occur here when $q=1$ and $k\geq 2$, and satisfy the recurrence relation
		\begin{equation}\label{eq01}
			F_{n}^{(k)}=F_{n-1}^{(k)}+\cdots + F_{n-k}^{(k)}, \ \forall \ n\geq 2,
		\end{equation}
		with the $k$ initial conditions given by $F_{-(k-2)}^{(k)}=F_{-(k-3)}^{(k)}=\cdots =F_0^{(k)}=0$ and $F_1^{(k)}=1$.
		
		Another generalization of Fibonacci numbers is the so-called sequence of $q$-Fibonacci numbers or $q$-weight Fibonacci numbers that appear for $q\geq1$ and $k=2$ in \eqref{eq04}. These sequences are defined by
	\begin{equation}\label{eq02}
		F_{q,n}=qF_{q,n-1}+F_{q,n-2}\qquad\forall\quad  n\geq 2 
	\end{equation}
	with initial conditions $F_{q,0}=0$ and $F_{q,1}=1$ (see \cite{catarino,falcon}). Probably, the best known example of these sequences is the sequence of Pell numbers given by the above recurrence taking $q=2$.
	
	In a recent paper, Kili\c{c} and Ta\c{s}ci \cite{kilicc2006generalized} studied the sequences of order-$k$ Pell numbers given by 
	\begin{equation}\label{eq03}
			P^{(k)}_{n}=2P^{(k)}_{n-1}+P^{(k)}_{n-2}+\cdots+P^{(k)}_{n-k}\qquad\forall\quad n\geq 2
		\end{equation}
		with $P_{-(k-2)}^{(k)}=P_{-(k-3)}^{(k)}=\cdots =P_0^{(k)}=0$ and $P_1^{(k)}=1$.
	Note that in \eqref{eq04} with $q=2$ and $k\geq 3$, we have this $k$-Pell sequence.
	
	In \cite{bravo2020combinatorial}, Bravo, Herrera and Ramírez presented some combinatorial interpretations for these $k$-Pell numbers while in \cite{bravo2021generalization}, Bravo, Herrera and Luca determined a Binet-style formula for these sequences and proved several results on the asymptotic behavior of the dominant root of their characteristic equation. Furthermore, as shown in \cite{bravo2021generalization}, the sequences of order-$k$ Pell numbers have a good asymptotic behavior, as well as the sequences of $k$-generalized Fibonacci numbers. 
	
		Besides that, in the case $q=3$, we have the firsts terms of the sequences of $(3,k)$-generalized Fibonacci numbers are given by
	\begin{eqnarray*}
		\begin{tabular}{c|cccccccccc}
			\rule[-1ex]{0pt}{2.5ex} \diagbox{k}{n} & 1 & 2 & 3 & 4 & 5 & 6 & 7 & 8 & 9 & \ldots \\
			\hline
			\rule[-1ex]{0pt}{2.5ex} 2 & 1 & 3 & 10 & 33 & 109 & 360 & 1189 & 3927 & 12970 &  \ldots \\
			\hline
			\rule[-1ex]{0pt}{2.5ex} 3 & 1 & 3 & 10 & 34 & 115 & 389 & 1316 & 4452 & 15061 &  \ldots \\
			\hline
			\rule[-1ex]{0pt}{2.5ex} 4 & 1 & 3 & 10 & 34 & 116 & 395 & 1345 & 4580 & 15596 &\ldots \\
			\hline
			\rule[-1ex]{0pt}{2.5ex} 5 & 1 & 3 & 10 & 34 & 116 & 396 & 1351 & 4609 & 15724 &\ldots \\
		\end{tabular}
	\end{eqnarray*}
	and, for $q=4$, we have the list of the first terms of the sequences of  $(4,k)$-generalized Fibonacci numbers given by
	\begin{eqnarray*}
		\begin{tabular}{c|cccccccccc}
			\rule[-1ex]{0pt}{2.5ex} \diagbox{k}{n} & 1 & 2 & 3 & 4 & 5 & 6 & 7 & 8 & 9 & \ldots \\
			\hline
			\rule[-1ex]{0pt}{2.5ex} 2 & 1 & 4 & 17 & 72 & 305 & 1292 & 5473 & 23184 & 98209 & \ldots \\
			\hline
			\rule[-1ex]{0pt}{2.5ex} 3 & 1 & 4 & 17 & 73 & 313 & 1342 & 5754 & 24671 & 105780 & \ldots \\
			\hline
			\rule[-1ex]{0pt}{2.5ex} 4 & 1 & 4 & 17 & 73 & 314 & 1350 & 5804 & 24953 & 107280 & \ldots \\
				\hline
			\rule[-1ex]{0pt}{2.5ex} 5 & 1 & 4 & 17 & 73 & 314 & 1351 & 5812 & 25003 & 132565 & \ldots \\
		\end{tabular}
	\end{eqnarray*}
	
    Since the definition of $(q,k)$-generalized Fibonacci sequence and basic properties about its behavior are the first step to prove new results  about Diophantine equations involving these sequences, the main goal of this paper is to extend the main results on asymptotic behavior to the sequences defined in \eqref{eq04}.  So, we have the following main theorem that we shall prove in the section 3.

	\begin{theorem}[Main Theorem]\label{theo1}
		Let $(F^{(k)}_{q,n})$ be the sequence of $(q,k)$-generalized Fibonacci numbers with $k\geq2$, $q\geq3$ and $n\geq2-k$. Then
		\begin{enumerate}[(a)]
			\item
			\[
			F_{q,n}^{(k)}=\sum_{i=1}^kg_{q,k}(\gamma_i)\gamma_i^n,
			\]
			where $\gamma_1,\gamma_2,\ldots,\gamma_k$ are roots of characteristic polynomial $\Phi_{q,k}(t)$, given by
			\[
			\Phi_{q,k}(t)=t^k-qt^{k-1}-t^{k-2}-\cdots-t-1
			\]
			and
			\[
			g_{q,k}(x):=\frac{x-1}{(k+1)x^2-(q+1)kx+(q-1)(k-1)};
			\]
			\item
			\[
    			|F^{(k)}_{q,n}-g_{q,k}(\gamma)\gamma^n|\leq\frac{1}{q},
			\]
		     where $\gamma$ is the dominant root of $\Phi_{q,k}(t)$. Moreover,
			\begin{equation*}
				\gamma^{n-2}<\gamma^{n-1}\left(\frac{q-1}{q}\right)<F_{q,n}^{(k)}<\gamma^{n-1}\left(\frac{q+2}{q}\right)<\gamma^{n},
			\end{equation*}
			for all $n\geq1$.
		\end{enumerate}
	\end{theorem}

         We remark that the case $q=1$ of this theorem was proved by Dresden and Du in \cite{dresden} and the case $q=2$ was proved by Bravo, Herrera and Luca in \cite{bravo2021generalization}. Therefore, in this paper we shall consider the case $q\geq3$.
	
	\section{Preliminary Results}
	
	For a detailed study of the sequences of  $(q,k)$-generalized Fibonacci numbers, $(F^{(k)}_{q,n}):=(F^{(k)}_{q,n})_{n\geq -(k-2)}$, we consider the characteristic polynomial defined by  
	\begin{equation}\label{eq05}
		\Phi_{q,k}(t)=t^k-qt^{k-1}-t^{k-2}-\cdots-t-1
	\end{equation}
	and the auxiliary function
	\begin{equation}\label{eq06}
		h_{q,k}(t)=(t-1)\Phi_{q,k}(t)=t^{k+1}-(q+1)t^{k}+(q-1)t^{k-1}+1.
	\end{equation}

	Since ($F^{(k)}_{q,n})$ is a linear recurrence of order $k$ with characteristic polynomial given by \eqref{eq05} and this polynomial divides the auxiliary function \eqref{eq06}, we deduce that $(F^{(k)}_{q,n})$ is also a linear recurrence of order $k+1$ with characteristic polynomial $h_{q,k}(t)$. Hence, we obtain our first identity involving the sequences of  $(q,k)$-generalized Fibonacci number given by the following theorem.
	
	\begin{theorem}\label{theo2}
		Let $k\geq 2$ and $q\geq 3$ be integer number. Then
		\begin{equation}\label{eq07}
			F^{(k)}_{q,n}=(q+1)F^{(k)}_{q,n-1}-(q-1)F^{(k)}_{q,n-2}-F^{(k)}_{q,n-k-1}\qquad \mbox{for all}\quad  n\geq 3.
		\end{equation}
	\end{theorem}

	Now, we can use Theorem \ref{theo2} and the fact that, for $-(k-2)\leq n\leq 0$, $F^{(k)}_{q,n}=0$ to obtain the first terms in the sequences of $(q,k)$-generalized Fibonacci numbers. Let $q\geq 3$ be a integer number and $(U_{q,n})_{n\geq 1}$ the sequence given by recurrence 
	\[
	U_{q,n}=(q+1)U_{q,n-1}-(q-1)U_{q,n-2}
	\]
	for all $n\geq 3$ with $U_{q,1}=1$ and $U_{q,2}=q$. 
	Then we have that
	
	\begin{equation*}
		U_{q,n}=\frac{((q-3)+\sqrt{q^2-2q+5})\alpha^{n}_q+((3-q)+\sqrt{q^2-2q+5})\beta^{n}_q}{2(q-1)\sqrt{q^2-2q+5}}
	\end{equation*}
	for all integer number $ n\geq 1$, where $\alpha_q$ and $\beta_q$ are the roots of $$t^2-(q+1)t+(q-1)=0,$$ given by 
	\[	
	\alpha_q=\frac{(q+1)+\sqrt{q^2-2q+5}}{2},\ \beta_q=\frac{(q+1)-\sqrt{q^2-2q+5}}{2}.
	\]
	
	In what follows, it will be important to observe that $q<\alpha_q<q+1$ and $0<\beta_q<1$, for all integer number $q\geq 3$. Moreover, let us consider the sequence given by recurrence 
	$$V_{q,n}=(q+1)V_{q,n-1}-(q-1)V_{q,n-2}$$
	for all $n\geq 3$ with $V_{q,1}=1$ and $V_{q,2}=q+1$. So, it is not difficult to prove that $F^{(k)}_{q,n}\leq U_{q,n}$, more precisely, we have the following theorem which can be prove by induction
	
	\begin{theorem}\label{theo3}
		Let $k\geq 2$ and $q\geq 3$ be integer number. Then
		\begin{equation}\label{eq08}
			F^{(k)}_{q,n}=U_{q,n} 
		\end{equation}
		for all $1\leq n\leq k+1$ and
		\begin{equation}\label{eq09}
			F^{(k)}_{q,n}=U_{q,n}-\sum_{j=1}^{n-k-1}V_{q,j}F_{q,n-k-j}^{(k)}
		\end{equation}
		for all $n\geq k+2$.
	\end{theorem}
	
	
	Another important fact about the sequence $(F_{q,n}^{(k)})_{n\geq0}$ is that its generating function
	\[
	f_{q,k}(x)=\sum_{n=0}^{\infty} F^{(k)}_{q,n}x^n,
	\]
	is given by
	\begin{equation}
		f_{q,k}(x)=\frac{x}{1-qx-x^2-\ldots-x^k}.
	\end{equation}

	
	In the next section, in order to prove the Main Theorem, we shall determine the Binet-style formulas and study the asymptotic behavior of dominant root of characteristic equation to the sequences of $(q,k)$-generalized Fibonacci numbers.
	
	\section{Proof of Main Theorem}

	\subsection{Binet-Style Formula.}
		Kalman in \cite{kalman1982generalized} proved that if $(u_n)_{n\geq0}$ is a linear recurrence sequence of order $k\geq2$  satisfies the recurrence  
		\[
		u_{n+k}=c_{k-1}u_{n+k-1}+c_{k-2}u_{n+k-2}+\cdots+c_1u_{n+1}+c_0u_{n}
		\]
		for all $n\geq0$, with initial condition $u_0=u_1=\cdots=u_{k-2}=0$ and $u_{k-1}=1$, where $c_0,c_1,\ldots,c_{k-1}$ are constant. Then
		\[
		u_n=\sum_{i=1}^{k}\frac{\alpha_i^n}{P'(\alpha_i)}
		\]
		where $P(t)=t^k-c_{k-1}t^{k-1}-\cdots-c_1t-c_0$ is the characteristic polynomial of $(u_n)_{n\geq0}$ and $\alpha_1,\alpha_2,\ldots,\alpha_k$ are the roots of $P(t)$. Taking $u_n=F^{(k)}_{q,n-(k-2)}$ for all $n\geq0$, we have that $P(t)=\Phi_{q,k}(t)$ and 
		
			\[
			F^{(k)}_{q,n}=\sum_{i=1}^{k}\frac{\gamma_i^{n+(k-2)}}{\Phi'_{q,k}(\gamma_i)}
			\]
		where $\gamma_1,\gamma_2,\ldots,\gamma_k$ are the roots of $\Phi_{q,k}(t)$. Now, consider the auxiliary function \eqref{eq06} given by 
		\[
		h_{q,k}(t)=(t-1)\Phi_{q,k}(t)=t^{k+1}-(q+1)t^{k}+(q-1)t^{k-1}+1.
		\]
		Using that 
		\[
		\Phi'_{q,k}(t)=\frac{h'_{q,k}(t)(t-1)-h_{q,k}(t)}{(t-1)^2}
		\]
		and $h_{q,k}(\gamma_i)=0$ for all $0\leq i\leq k$, we obtain
		\begin{eqnarray*}
			\Phi'_{q,k}(\gamma_i)&=&\frac{(k+1)\gamma_i^k-(q+1)k\gamma_i^{k-1}+(q-1)(k-1)\gamma_i^{k-2}}{\gamma_i-1}.
		\end{eqnarray*}
		Hence,
		\[
		F_{q,n}^{(k)}=\sum_{i=1}^kg_{q,k}(\gamma_i)\gamma_i^n
		\]
		where
		\[
		g_{q,k}(x):=\frac{x-1}{(k+1)x^2-(q+1)kx+(q-1)(k-1)},
		\]
		and this proves the item (a) of Theorem \ref{theo1}.
	
	\subsection{Asymptotic Behavior}
	For integers numbers $k\geq 2$, $q\geq 3$ and $n\geq2-k$, we define $E^{(k)}_{q,n}$ as the error of the approximation of the $n$th $(q,k)$-generalized Fibonacci number with the dominant term of the Binet-style formula of $F^{(k)}_{q,n}$, i.e.,
	\begin{equation}\label{eq3.4}
		E^{(k)}_{q,n}=F^{(k)}_{q,n}-g_{q,k}(\gamma)\gamma^{n},
	\end{equation}
	for $\gamma$ the dominant root of $\Phi_{q,k}$.
	
	We remark that given a polynomial $f(x)\in\mathbb{R}[x]$, the set of all possible linear recurrence sequences of real numbers having the characteristic equation $f(x)=0$ is a vector space over real numbers. Since $F^{(k)}_{q,n}$ and $(\gamma^n)_{n\geq 0}$ satisfy the characteristic equation $\Phi_{q,k}(x)=0$, it follows from \eqref{eq3.4} that $E^{(k)}_{q,n}$ satisfies the same recurrence relation as $(q,k)$-generalized Fibonacci sequence. That is, for all integer $k\geq2$ we have that
		\[
	E_{q,n}^{(k)}=qE_{q,n-1}^{(k)}+E_{q,n-2}^{(k)}+\cdots+E_{q,n-k}^{(k)}\qquad\text{for all}\quad n\geq2,
	\]
	moreover,
    \begin{equation}\label{eq2.7}
		E_{q,n}^{(k)}=(q+1)E_{q,n-1}^{(k)}-(q-1)E_{q,n-2}^{(k)}-E_{q,n-k-1}^{(k)}.
	\end{equation}
	
	Using the fact that $\lim_{n\to\infty}|\gamma_i|^n=0$ for $2\leq i\leq k$ and taking into account that
	\[
	|E_{q,n}^{(k)}|\leq \sum_{j=2}^{k}|g_{q,k}(\gamma_j)||\gamma_j|^n,
	\]
	we also deduce that
	\begin{equation}\label{eq2.8}
		\lim_{n\to\infty}|E_{q,n}^{(k)}|=0.
	\end{equation}

		Our aim is to estimate $|E_{q,n}^{(k)}|$. For this, we remark that  Wu and Zang in \cite{wu2013reciprocal} proved that for all integer numbers $a_1\geq a_2\geq\cdots\geq a_m\geq1$ with $m\geq2$ the polynomial 
		\[
		f(x)=x^m-a_1x^{m-1}-a_2x^{m-2}-\cdots-a_1x-a_m
		\]  
		has exactly one positive real zero $\alpha$ with $a_1<\alpha<a_1+1$ and the others $m-1$ zeros of $f(x)$ lie within the unit circle in the complex plane. Thus, for all $q\geq3$ the characteristic polynomial $\Phi_{q,k}(t)$ has a dominant root $q<\gamma<q+1$ and the others roots are into the unit circle. 
		
		In what follows, fixed an integer number $q\geq 3$, we will indicate by $\gamma_{k}$ the dominant root of $\Phi_{q,k}(t)$. In order to prove the item (b) of Theorem \ref{theo1}, we shall introduce some auxiliary results about $\gamma_{k}$.
		
		\begin{lemma}\label{lemma1}
			Let $q\geq 3$ be an integer number fixed, $\gamma_{l}$ and $\gamma_{k}$ the dominant roots of $\Phi_{q,l}(t)$ and $\Phi_{q,k}(t)$, respectively. Then
			\begin{enumerate}[(i)]
				\item for $l>k$, we have that $\gamma_{l}>\gamma_{k}$;
				\item if $\alpha_q$ is the dominant root of $t^2-(q+1)t+(q-1)=0$, then $$\alpha_q\left(1-\frac{1}{q^k}\right)<\gamma_{k}<\alpha_q.$$
			\end{enumerate}
		\end{lemma}
		
		\begin{proof}
			In order to prove the item (i), we proceed by contradiction. Let us assuming that $\gamma_{k}\geq\gamma_{l}$. Thus, $\gamma_{k}^{-i}\leq \gamma_{l}^{-i}$ holds for all $i\geq 1$. Taking into account that $\Phi_{q,k}(\gamma_{q,k})=0$, we get
			\[
			\gamma_{k}^k=q\gamma_{k}^{k-1}+\gamma_{k}^{k-2}+\cdots+\gamma_{k}+1,
			\]
			and the same conclusion remains valid for $\gamma_l$. Since $k<l$ we have that 
			\begin{eqnarray*}
				1&=&\dfrac{q}{\gamma_{k}}+\dfrac{1}{\gamma_{k}^2}+\cdots+\dfrac{1}{\gamma_{k}^k}\\
				&<&\dfrac{q}{\gamma_{l}}+\dfrac{1}{\gamma_{l}^2}+\cdots+\dfrac{1}{\gamma_{l}^k}+\dfrac{1}{\gamma_{l}^{k+1}}+\cdots+\dfrac{1}{\gamma_{l}^l}=1
			\end{eqnarray*}
			which is a contradiction. Thus, we conclude that $\gamma_{l}>\gamma_{k}$ and this proves the item (i).
			
			Let us consider the item (ii). Note that the auxiliary function given by \eqref{eq06} can be writing in the form
			\[
			h_{q,k}(t)=t^{k-1}(t^2-(q+1)t+(q-1))+1,
			\]
			  moreover,  $\alpha_q$ is a root of $t^2-(q+1)t+(q-1)$. Thus, $h_{q,k}(\alpha_q)=1$ and 
			\[
			\Phi_{q,k}(\alpha_q)=\frac{h_{q,k}(\alpha_q)}{\alpha_q-1}=\frac{1}{\alpha_q-1}>0.
			\]
			In the other hand, we have that $\alpha_q>q$ and 
			\[
			\Phi_{q,k}(q)=-q^{k-2}-q^{k-3}-\cdots-q-1<0.
			\]
			Since $\gamma_k$ is the only root of $\Phi_{q,k}(x)$ such that $\gamma_{k}>1$, we obtain that $q<\gamma_{k}<\alpha_q$.
			
			By hypothesis, $$\alpha_q^2-(q+1)\alpha_q+(q-1)=0,$$
			and evaluating $h_{q,k}(t)$ at $\gamma_{k}$, we have  $h_{q,k}(\gamma_{k})=0$ and then 
			\begin{eqnarray*}
				\gamma_{k}^2-(q+1)\gamma_{k}+(q-1)= \frac{-1}{\gamma_{k}^{k-1}}.
			\end{eqnarray*}
			Subtracting the two expressions above and rearranging some terms, we obtain
			\[
			(\alpha_q-\gamma_{k})(\alpha_q+\gamma_{k}-(q+1))=\frac{1}{\gamma_{k}^{k-1}}.
			\]
			Since $\alpha_q>\gamma_{k}>q$ and $(\alpha_q+\gamma_{k}-(q+1))>q/\alpha_q$, we get $\alpha_q-\gamma_{k}<\alpha_qq^{-k}$. Hence, 
			\[
			\gamma_{k}>\alpha_q\left(1-\frac{1}{q^k}\right)
			\]
			and this concludes the proof of (ii).
		\end{proof}

		Now, let us consider that $q\geq3$ and study the rational function $g_{q,k}$ that appears in the Binet-style formula given by
		\[
		g_{q,k}(x)=\frac{x-1}{(k+1)x^2-(q+1)kx+(q-1)(k-1)}.
		\]
		Since $\alpha_q$ is a root of $x^2-(q+1)x+(q-1)$, we get
		\[
		g_{q,k}(\alpha_q)=\frac{\alpha_q-1}{\alpha_q^2-(q-1)}.
		\]
		In particular, after some calculations, we have the estimate 
            \begin{equation}\label{eq34}
                \frac{1}{q+1}<g_{q,k}(\alpha_q)<\frac{1}{q},
            \end{equation}
            where we use that $3\leq q<\alpha_q<q+1$. Moreover, we remark that $g_{q,k}$ has vertical asymptote in 
		\begin{equation}\label{eq35}
		    c_{q,k}:=\frac{(q+1)k+\sqrt{k^2(q^2-2q+5)+4(q-1)}}{2(k+1)}
		\end{equation}
		and is positive and continuous in $(c_{q,k},\infty)$. Further,
		\[
		g_{q,k}'(x)=-\frac{(k+1)(x-1)^2+q+k-2}{[(k+1)x^2-(q+1)kx+(q-1)(k-1)]^2},
		\]
		
		is negative in $(c_{q,k},\infty)$, so $g_{q,k}(x)$ is decreasing in the same interval. Using this we can prove the following technical lemma.
		
		\begin{lemma}\label{lemma2}
			Let $\gamma_k$ be a dominant root of $\Phi_{q,k}(t)$, $k\geq2$ and $q\geq3$ integers. Then
			\[
			\frac{1}{q+1}<g_{q,k}(\gamma_{k})<\frac{1}{q}.
			\]
		\end{lemma}
		
		\begin{proof}
		In order to prove this lemma, we consider three cases. Firstly, we consider the case $k=2$. In this case we have that
			\[
			g_{q,2}(x)=\frac{x-1}{3x^2-2(q+1)x+(q-1)}
			\]
			and $\gamma_2$ is the biggest root of $t^2-qt-1=0$ given by $\gamma_2=(q+\sqrt{q^2+4})/2$. Since $q\geq3$, we get 
			\[
			\frac{1}{q+1}<g_{q,2}(\gamma_2)<\frac{1}{q}.
			\]
            
               Now, let us consider the the case $3\leq k\leq q$. Here we use the value $c_{q,k}$ defined in \eqref{eq35} and we remark that it can be rewritten as
			\[
			c_{q,k}=\frac{(q+1)+\sqrt{q^2-2q+5+4k^{-2}(q-1)}}{2}\left(1-\frac{1}{k+1}\right).
			\]
			Thus, it is easily seen that $c_{q,k}\to \alpha_q$ as $k\to \infty$. Since 
			$$\sqrt{q^2-2q+5+4k^{-2}(q-1)}< \sqrt{q^2-2q+5}+\frac{2\sqrt{q-1}}{k},$$ we get
			\begin{equation}\label{eq13}
			   c_{q,k}<\left(\alpha_q+\frac{\sqrt{q-1}}{k}\right)\left(1-\frac{1}{k+1}\right). 
			\end{equation}

                Therefore, we obtain that
			\[
			c_{q,k}<\alpha_q-\frac{1}{q(q-1)}<\alpha_q\left(1-\frac{1}{q^{k}}\right)<\gamma_{k}<\alpha_q,
			\]
			where the last inequalities are given by Lemma \ref{lemma1}.
			
			Denoting by $G_{q,k}$ the value $g_{q,k}(\alpha_q-1/(q(q-1)))$,	since $g_{q,k}(x)$ is decreasing in $(c_{q,k},\infty)$ and by \eqref{eq34}, we have that
			\[
			\frac{1}{q+1}<g_{q,k}(\alpha_q)<g_{q,k}(\gamma_{k})< G_{q,k}.
			\]
			Thus, after some calculations, we obtain that 
			\begin{eqnarray*}
		G_{q,k}&=&\frac{\alpha_q-1-\frac{1}{q(q-1)}}{q\alpha_q+(\alpha_q-q)-q+2\left(1-\frac{\alpha_q(k+1)}{q(q-1)}\right)+\frac{k(q+1)}{q(q-1)}+\frac{k+1}{q^2(q-1)^2}}\\
		&<&\frac{\alpha_q-1-\frac{1}{q(q-1)}}{q\left(\alpha_q-1-\frac{1}{q(q-1)}\right)},
			\end{eqnarray*}
		where we use that $q\geq k\geq3$ and $\alpha_q=(q+1+\sqrt{q^2-2q+5})/2$. Hence, we get $g_{q,k}(\gamma_{k})<1/q$.
			
			Finally, we consider the case $k\geq q+1$ and observe that by \eqref{eq13} we have that  
			\[
			c_{q,k}<\alpha_q-\frac{1}{k(k+1)}.
			\]
			Since $k\geq q+1\geq \alpha_q$ and $q\geq 3$, we get $k(k+1)\alpha_q< q^k$. Moreover, using that $k^2(k+1)<3^k$ for all $k\geq4$, we obtain
			\[
			c_{q,k}<\alpha_q-\frac{1}{k(k+1)}<\alpha_q\left(1-\frac{1}{q^{k}}\right)<\gamma_{k}<\alpha_q.
			\]
			  Thus, we use again the $g_{q,k}(x)$ is decreasing in $(c_{q,k},\infty)$ and we get
			\[
			g_{q,k}(\alpha_q)<g_{q,k}(\gamma_{k})<g_{q,k}\left(\alpha_q-\frac{1}{k(k+1)}\right).
			\]
			In particular, by \eqref{eq34},
			\[
			\frac{1}{q+1}<g_{q,k}(\gamma_{k}).
			\]
			
			In the other hand, let us denote the value $g_{q,k}(\alpha_q-1/(k(k+1)))$ just by $\overline{G}_{q,k}$. Then, we obtain that
			\begin{eqnarray*}
			\overline{G}_{q,k}&=&\frac{\alpha_q-1-\frac{1}{k(k+1)}}{q\alpha_q+(\alpha_q-q)-q+2\left(1-\frac{\alpha_q}{k}\right)+\frac{k^2(q+1)+1}{k^2(k+1)}}\\
			&<&\frac{\alpha_q-1-\frac{1}{k(k+1)}}{q\left(\alpha_q-1-\frac{1}{k(k+1)}\right)},
			\end{eqnarray*}
			 where we use that $k\geq q+1>\alpha_q>q\geq3$. Therefore, we obtain that $g_{q,k}(\gamma_{k})<1/q$ and this completes the proof of the lemma.
		\end{proof}
		
		In order to prove the item (b) of Theorem \ref{theo1} which state that if $q\geq 3$ and $k\geq 2$, then
		\[
		|E_{q,n}^{(k)}|\leq\frac{1}{q}
		\]
		for all $n\geq 2-k$. Note that, by initial conditions of $(F_{q,n}^{(k)})$, we have $F^{(k)}_{q,n}=0$ for all $2-k\leq n\leq 0$. Hence 
		$$E_{q,n}^{k}=-g_{q,k}(\gamma)\gamma^n$$
		for all $2-k\leq n\leq 0$. Now let us suppose $n=0$, in this case, by Lemma \ref{lemma2}, we get 
		$$|E^{(k)}_{q,0}|=g_{q,k}(\gamma)<1/q.$$ 
		Moreover, if $2-k\leq n\leq -1$, then $\gamma^n\leq\gamma^{-1}<1$ and again we obtain that 
		$$g_{q,k}(\gamma)\gamma^n\leq g_{q,k}(\gamma)<1/q$$ 
		for all $k\geq 2$.
		
		By Lemma \ref{lemma2}, we have that $\gamma/(q+1)<g_{q,k}(\gamma)\gamma<\gamma/q$. Using that $q<\gamma<q+1$, we get $1-1/(q+1)<g_{q,k}(\gamma)\gamma<1+1/q$. Thus, if $n=1$, then $F_{q,1}^{(k)}=1$ and 
		\[
		-\frac{1}{q}<1-g_{q,k}(\gamma)\gamma<\frac{1}{q+1}.
		\]
		  Hence, we obtain that $|E_{q,1}^{k}|\leq1/q$. 
		
		Now, for the sake of contradiction, assume that $|E_{q,n}^{(k)}|>1/q$ for some integer $n\geq2$. Let $n_0$ be the smallest positive integer with this property. Since $|E_{q,n_0-1}^{(k)}|\leq1/q$ and $|E_{q,n_0-k}|\leq1/q$ we get $$|(q-1)E_{q,n_0-1}^{(k)}+E_{q,n_0-k}|\leq 1.$$ According to \eqref{eq2.7}  
		\[
		E_{q,n_0+1}^{(k)}=(q+1)E_{q,n_0}^{(k)}-((q-1)E_{q,n_0-1}^{(k)}+E_{q,n_0-k}^{(k)})
		\]
		and so
		\[
		|E_{q,n_0+1}^{(k)}|\geq(q+1)|E_{q,n_0}^{(k)}|-|(q-1)E_{q,n_0-1}^{(k)}+E_{q,n_0-k}^{(k)}|.
		\]
		Hence
		\[
		|E_{q,n_0+1}^{(k)}|-|E_{q,n_0}^{(k)}|\geq q|E_{q,n_0}^{(k)}|-|(q-1)E_{q,n_0-1}^{(k)}+E_{q,n_0-k}^{(k)}|>0
		\]
		giving
		\[
		|E_{q,n_0+1}^{(k)}|>|E_{q,n_0}^{(k)}|.
		\]
		Since $n_0-k+1<n_0$, we infer that $|E^{(k)}_{q,n_0-k+1}|\leq \frac{1}{q}< |E_{q,n_0}^{(k)}| <|E_{q,n_0+1}^{(k)}|$ and therefore $|(q-1)E_{q,n_0}^{(k)}+E^{(k)}_{q,n_0-k+1}|<q|E_{q,n_0+1}^{(k)}|$. Thus,
		\[
		|E_{q,n_0+2}^{(k)}|\geq(q+1)|E_{q,n_0+1}^{(k)}|-|(q-1)E_{q,n_0}^{(k)}+E_{q,n_0-k+1}^{(k)}|
		\]
		and we obtain that $|E_{q,n_0+2}^{(k)}|>|E_{q,n_0+1}^{(k)}|$.
		
		Thus, suppose that $|E^{(k)}_{q,n_0}|<|E^{(k)}_{q,n_0+1}|<\cdots<|E^{(k)}_{n_0+i-1}|$ for some integer $i\geq 4$. We have two cases according to whether $n_0+i-k-1<n_0$ or $n_0\leq n_0+i-k-1$. First, if $n_0+i-k-1<n_0$, then we get
		\[
		|E_{q,n_0+i-k-1}^{(k)}|\leq\frac{1}{q}<|E^{(k)}_{q,n_0}|<|E_{q,n_0+1}^{(k)}|<\cdots<|E^{(k)}_{n_0+i-1}|.
		\]
		In the other hand, if $n_0\leq n_0+i-k-1$, then we obtain that 
		$$|E_{q,n_0+i-k-1}^{(k)}|<|E_{q,n_0+i-1}^{(k)}|.$$ 
		In any case, we conclude that $|E_{q,n_0+i-k-1}^{(k)}|<|E_{q,n_0+i-1}^{(k)}|$. For this reason
		\[
		|(q-1)E^{(k)}_{q,n_0+i-2}+E^{(k)}_{q,n_0+i-k-1}|<q|E_{q,n_0+i-1}^{(k)}|.
		\]
		Using \eqref{eq2.7} again, we get
		\begin{eqnarray*}
			|E^{(k)}_{q,n_0+i}|&\geq& (q+1)|E^{(k)}_{n_0+i-1}|-|(q-1)E^{(k)}_{q,n_0+i-2}+E^{(k)}_{q,n_0+i-k-1}|\\
			&>& |E_{q,n_0+i-1}^{(k)}|.  	
		\end{eqnarray*}
		
		Therefore, $|E^{(k)}_{q,n_0}|<|E_{q,n_0+1}^{(k)}|<\cdots<|E^{(k)}_{n_0+i-1}|<|E^{(k)}_{n_0+i}|$ contradicting \eqref{eq2.8} which says that the error must eventually go to $0$. Hence, we conclude that $|E_{q,n}^{(k)}|<1/q$ for all integer $n\geq2-k$.  	
		
	Now, let us consider the exponential growth of $(q,k)$-generalized Fibonacci sequence. Note that, by the first part of item (b) of Theorem \ref{theo1}, if $\gamma=\gamma_{q,k}$ with $q\geq 3$ and $k\geq2$, then
	\[
	g_{q,k}(\gamma)\gamma^n-\frac{1}{q}<F_{q,n}^{(k)}<g_{q,k}(\gamma)\gamma^n+\frac{1}{q}.
	\]
	Since $1/(q+1)<g_{q,k}(\gamma)<1/q$, we obtain that
	\[
	\frac{\gamma^n}{q+1}-\frac{1}{q}<F_{q,n}^{(k)}<\frac{\gamma^n}{q}+\frac{1}{q}.
	\]
	
	Hence, if $k\geq 2$ and $q\geq 3$, then we have that
	\[
	\gamma^{n-2}<\gamma^{n-1}\left(\frac{q-1}{q}\right)<F_{q,n}^{(k)}<\gamma^{n-1}\left(\frac{q+2}{q}\right)<\gamma^{n},
	\]
    for all $n\geq1$, where we use that $q<\gamma<q+1$. Finally, this completes the proof of item (b) of Theorem \ref{theo1}.
    \qed



	
	\bibliographystyle{amsplain}

\end{document}